\documentclass[12pt, twoside]{article}
\usepackage{amsmath,amsthm,amssymb}
\usepackage[all]{xy}

\usepackage{times}
\usepackage{enumerate}

\def\titlerunning#1{\gdef\titrun{#1}}
\makeatletter
\def\author#1{\gdef\autrun{\def\and{\unskip, }#1}\gdef\@author{#1}}
\def\address#1{{\def\and{\\\hspace*{18pt}}\renewcommand{\thefootnote}{}%
\footnote {#1}}%
\markboth{\autrun}{\titrun}}
\makeatother
\def\email#1{e-mail: #1}
\def\subjclass#1{{\renewcommand{\thefootnote}{}%
\footnote{\emph{Mathematics Subject Classification (2010):} #1}}}
\def\keywords#1{\par\medskip
\noindent\textbf{Keywords.} #1}

\newtheorem{thm}{Theorem}[section]
\newtheorem{cor}[thm]{Corollary}
\newtheorem{lem}[thm]{Lemma}

\newtheorem{proposition}[thm]{Proposition}

\theoremstyle{definition}
\newtheorem{defin}[thm]{Definition}
\newtheorem{rem}[thm]{Remark}
\newtheorem{exa}[thm]{Example}

\numberwithin{equation}{section}

\newcommand{\colim}{\mathop{\mathrm{colim}}}

\newcommand{\Symm}{\ensuremath{\mathrm{Symm}}}
\newcommand{\NSymm}{\ensuremath{\mathrm{NSymm}}}
\newcommand{\QSymm}{\ensuremath{\mathrm{QSymm}}}
\newcommand{\Id}{\ensuremath{\mathrm{Id}}}

\newcommand{\A}{\ensuremath{\mathcal{A}}}
\newcommand{\C}{\ensuremath{\mathbb{C}}}

\newcommand{\Q}{\ensuremath{\mathbb{Q}}}
\newcommand{\coefZ}{\ensuremath{\mathbb{Z}}}
\newcommand{\Z}{\ensuremath{\mathcal{Z}}}
\renewcommand{\d}{\ensuremath{\partial}}
\newcommand{\ul}[1]{\underline{#1}}

\newcommand{\lra}{\longrightarrow}
\newcommand{\dr}[3]{\ensuremath{#1\stackrel{#2}
{\longrightarrow}#3}}
\newcommand{\ddr}[5]{\ensuremath{#1\stackrel{#2}
{\longrightarrow}#3\stackrel{#4}{\longrightarrow}#5}}

\begin{document}


\baselineskip=17pt


\titlerunning{Hopf algebras and homotopy invariants}

\title{Hopf algebras and homotopy invariants}

\author{Victor Buchstaber 
\and 
Jelena Grbi\'{c}}

\date{}

\maketitle

\address{V. Buchstaber: Steklov Mathematical Institute of Russian Academy of Science, Moscow, Russia; \email{buchstab@mi.ras.ru}
\and
J. Grbi\' c: School of Mathematics, University of Southampton, Southamtpon, UK; \email{J.Grbic@soton.ac.uk}}



\begin{abstract}
In this paper we explore new relations between Algebraic Topology and the theory of Hopf Algebras. For an arbitrary topological space $X$, the loop space homology $H_*(\Omega\Sigma X; \coefZ)$ is a Hopf algebra. We introduce a new homotopy invariant of a topological space $X$ taking for its value the isomorphism class (over the integers) of the Hopf algebra $H_*(\Omega\Sigma X; \coefZ)$. This invariant is trivial if and only if  the Hopf algebra $H_*(\Omega\Sigma X; \coefZ)$ is isomorphic to a Lie-Hopf algebra, that is, to a primitively generated Hopf algebra. We show that for a given $X$ these invariants are obstructions to the existence of a homotopy equivalence  $\Sigma X\simeq \Sigma^2Y$ for some space $Y$. 

Further on, using the notion of Hopf algebras, we establish new structural properties of the cohomology ring, in particular, of the cup product. For example, using the fact that the suspension of a polyhedral product $X$ is a double suspension, we obtain a strong condition on the cohomology ring structure of $X$. This gives an important application in toric topology. For an algebra to be realised as the cohomology ring of a moment-angle manifold $\Z_P$ associated to a simple polytope $P$, we found an obstruction in the Hopf algebra $H_*(\Omega\Sigma \Z_P)$.

In addition, we use homotopy decompositions to study particular Hopf algebras.

\end{abstract}

\section{Introduction}

The notion of Hopf algebras came from the study of the homology of topological spaces with additional structures, that is, with multiplication. The expression Hopf algebra was coined by Armand Borel in 1953, honouring the fundamental work of Heinz Hopf ~\cite{Hopf} on $\Gamma$-manifolds, later known as Hopf manifolds, which are manifolds equipped with a product operation.  Soon after, Hopf algebras became a subject worthy of extensive study on its own. Hopf algebras have nowadays numerous applications in various area of mathematics and theoretical physics.  In this paper we explore new relations between Algebraic Topology and the theory of Hopf Algebras.

A pointed topological space $X$ is an $H$-space if there is a continuous map  $\mu\colon X\times X\lra X$ called multiplication such that the base point acts as a left and right unit.
 $H$-spaces appear in mathematics all the time. The most classical examples are topological groups: spaces $X$ with a group structure such that both the multiplication map $\mu\colon X\times X\lra X$ and the inversion map $X\lra X$, $x\mapsto x^{-1}$, are continuous.
Other examples are based loop spaces, Eilenberg-MacLane spaces, finite $H$-spaces such as Lie groups,  and the 7 sphere, the homotopy fibre of an $H$-map and so on.

A significant breakthrough in the study of $H$-spaces was achieved by Hopf and Borel (see~\cite{Hopf, Borel}) who classified graded Hopf algebras over a field of characteristic 0 which can be realised as the cohomology rings of $H$-spaces, and consequently gave a necessary condition on a graded Hopf algebra to be the cohomology algebra of an $H$-space. They showed that the existence of the comultiplication with count in a graded Hopf algebra restricts the multiplicative structure considerably. Using this Hopf algebra classification, it is easy to see that, for example, $\C P^n$ is not an $H$-space. The Hopf and Borel theorems give necessary but not sufficient conditions for a space to be an $H$-space; for example,  although the cohomology of $BU(2)$ satisfies the conditions of the Hopf and Borel theorems, $BU(2)$ is not an $H$-space. It is worth noting that  Hopf and Borel studied possible multiplications of appropriate Hopf algebras and their results are based on algebra isomorphisms of Hopf algebras while not studying further their coalgebra structures.

Ever since Hopf and Borel's work, there has been an extensive development of the theory of graded connected Hopf algebras (see~\cite{Milnor-Moore}) using algebraic and topological tools.
Let $A$ be a Hopf algebra. Denote by $P(A)$ the primitives in $A$, by $I(A)$ the augmentation ideal of $A$, and by $Q(A)=I(A)/I(A)^2$ the indecomposables. Then $A$ is primitively generated if the natural map $P(A)\lra I(A)\lra Q(A)$ is onto.
Such a primitively generated Hopf algebra $A$ is called \emph{Lie-Hopf} algebra.

Let $R$ be a principal ideal domain and assume we work with a topological space $X$ such that $H^*(X\times X; R)\cong H^*(X;R)\otimes H^*(X;R)$. The diagonal map on $X$ induces the cup product in $H^*(X;R)$  and dually it is closely related to the comultiplication in the Hopf algebra $H_*(\Omega\Sigma X;R)$. We show that the cup product presents in certain ways obstructions to $H_*(\Omega\Sigma X;R)$ be a Lie-Hopf algebra. Our starting point is the Bott and Samelson theorem \cite{BS}  which in an analogous way to the Hopf and Borel classification of Hopf algebras gives necessary conditions for graded Hopf algebras to be realised as the homology algebras of loop-suspension spaces. Namely, they prove that the homology $H_*(\Omega\Sigma X;R)$ is the tensor algebra on the reduced homology of $X$. 

One of the classical problems in homotopy theory is the Milnor double suspension problem, posed in 1961,  in which he asked whether the double suspension of a homology sphere is homeomorphic to the standard sphere. This problem can be generalised in homotopy theory to the problem of describing all spaces which after being suspended twice become a triple suspension space.
Utilising the theory of graded connected Hopf algebras, in particular by studying the properties of  comultiplications, we look for a necessary condition for a topological space $X$ to have the property that $\Sigma X\simeq\Sigma^2 Y$ for some $Y$.   Our approach is similar in nature to the one used in the study of $H$-spaces in the sense that we explore certain properties of naturally arising Hopf algebras. The main result is the following theorem.

\begin{thm}
Let $X$ be a space such that $\Sigma X\simeq\Sigma^2 Y$ for some $Y$ and assume that $H_*(X;\coefZ)$ is torsion free. Then the Hopf algebra $H_*(\Omega\Sigma X; \coefZ)$ is isomorphic to a Lie-Hopf algebra.
\end{thm}

As will be discussed in Sections~\ref{examples} and \ref{polyhedral},  spaces with the property that $\Sigma X\simeq\Sigma^2 Y$ are not rare; they appear, for example, in the context of polyhedral product functors - which are functorial generalisations of moment-angle complexes $\Z_K$, complements of hyperplane arrangements, some simply connected 4-manifolds and so on. Note that the class of manifolds with this property seems to be quite narrow and hard to detect. Thus the result about moment-angle manifolds, that is $\Z_K$ when $K$ is the dual of the boundary of a simple polytope is of great importance.

\begin{cor}
Let $\Z_K$ be a moment-angle complex such that $H_*(\Z_K;\coefZ)$ is torsion free. Then the comultiplication in $H_*(\Z_K;\coefZ)$, dual to the ring structure of $H^*(\Z_K;\coefZ)$, defines a Hopf algebra structure on the tensor algebra $T(\widetilde{H}_*(\Z_K; \coefZ))$ which over $\coefZ$ is isomorphic to a Lie-Hopf algebra.
\end{cor}

The last result motivates us to establish new structural properties of the cohomology ring, in particular, of the cup product. Consider a topological space $X$ such that $H^*(X;\coefZ)$ is torsion free. Although, in general, a topological space $X$ might have a non-trivial cup product, once suspended, all the cup products are trivialised. However, looking at the Hopf algebra $H_*(\Omega\Sigma X;\coefZ)$, in particular at its coalgebra structure, we can recover the information about the cup product in the cohomology of $X$. This way we associate to the cup product in $H^*(X;\coefZ)$ a structural property that can be read off the Hopf algebra $H_*(\Omega\Sigma X;\coefZ)$. For example, if a topological space $X$ has the property that  $\Sigma X\simeq\Sigma^2 Y$ for some $Y$, then the Hopf algebra $H_*(\Omega\Sigma X; R)$ is isomorphic to a Lie-Hopf algebra meaning that there is a change of basis over $R$ in $T(\widetilde{H}_*(X;R))$ such that the reduced diagonal $\widetilde{\Delta}\colon H_*(X;R)\lra H_*(X;R)\otimes H_*( X;R)$  becomes trivial.  The dual statement gives a structural property of the cup product in $H^*(X;R)$.

The strength and beauty of this approach lies in the strong connection between algebra and topology.  In Section~\ref{applications} we detect some properties of the algebra of quasi-symmetric functions by identifying it with $H_*(\Omega\Sigma \C P^\infty ; R)$ and by applying homotopy theoretic decomposition methods.

\section{$S^nS$-spaces}

We start by introducing the notion of  $S^nS$-spaces.
\begin{defin}
For a given $n\in\mathbb N\cup\{0\}$, a CW complex $X$ is said to be an \emph{$S^nS$-space} if there is a CW complex $Y$ such that
$\Sigma^n X\simeq \Sigma^{n+1} Y$.
\end{defin}

It is seen readily from the definition that if $X$ is an $S^nS$-space, then it is also an $S^{n+1}S$-space. Thus a trivial example of $S^nS$-spaces is given by a suspension space $X$, that is,  $X\simeq \Sigma Y$ for some topological space $Y$.

Milnor considered a homology $3$-sphere $M^3$ with $\pi_1(M) \neq 0$ and asked whether the double suspension of $M^3$ is homeomorphic to $S^5$. This was partially proved in 1975 by
Edwards~\cite{Ed}, and in a sharper form in 1979 by
Cannon~\cite{Ca}. It provided the first example of a triangulated manifold which is not locally PL-homeomorphic to Euclidean space.
The solution of the Milnor double suspension problem provides us with a non-trivial example of an $S^2S$-space. In general, Cannon~\cite{Ca} proved that the double suspension $\Sigma^2 S^n_H$ of any homology $n$-sphere $S^n_H$ is homeomorphic to the topological sphere $S^{n+2}$,  therefore showing that a homology sphere is an $S^2S$-space.

Our study of $S^nS$-spaces begins by detecting operations under which the family of $S^nS$-spaces is closed.

If we start with a topological pair  $(X_1, X_2)$ of $S^nS$-spaces, a natural question to ask is whether the quotient space $X_1/X_2$ is an $S^nS$-space. The pair $(S^k, S^{k-1})$ where $S^{k-1}$ is imbedded as the equator in $S^k$ gives one of the simplest examples where the quotient $S^k/S^{k-1}\simeq S^k\vee S^k$ is an $S^0S$-space for $k\geq 1$.

 \begin{lem}
 Let  $(X_1, X_2)$ be a topological pair of $S^nS$-spaces. Then the quotient space $X_1/X_2$ is not in general an $S^nS$-space.
 \end{lem}
 \begin{proof} We prove the lemma by constructing an example. Let $E\lra S^2$ be a disc $D^2$ bundle associated to the Hopf bundle over $S^2$. Then $\d E=S^3$ and $E/\d E=\mathbb{C}P^2$. In this way we have constructed a topological pair $(E,\d E)$ of $S^1S$-spaces such that $E/\d E$ is not an $S^1S$-space.
\end{proof}

There are several operations under which the property of being an $S^nS$-space is preserved.
\begin{lem}
\label{wedge}
Let $X_1,\ldots,X_k$ be $S^nS$-spaces. Then the following spaces are $S^nS$-spaces:
\begin{enumerate}
\item $X_1\vee\ldots\vee X_k,$
\item $X_1\wedge\ldots\wedge X_k$,
\item $X_1\times\ldots\times X_k$.
\end{enumerate}
\end{lem}
\begin{proof}
Let us assume that $\Sigma^n X_i\simeq \Sigma ^{n+1}Y_i$ for some $Y_i$ and $i=1,\ldots, k$.
\begin{enumerate}
\item
The following homotopy equivalences
\[
\Sigma^n(X_1\vee\ldots\vee X_k)\simeq \Sigma^nX_1\vee\ldots\vee \Sigma^nX_k\simeq\Sigma^{n+1}(Y_1\vee\ldots\vee Y_k)
\]
give statement $(1)$.
\item  Statement $(2)$ is proved by the following homotopy equivalences
\[
\Sigma^n(X_1\wedge\ldots\wedge X_k)\simeq \Sigma^{n+1}Y_1\wedge X_2\wedge\ldots\wedge X_k\simeq\Sigma Y_1\wedge\Sigma^{n+1}Y_2\wedge\ldots\wedge X_k\simeq \Sigma^{n+k}Y_1\wedge\ldots\wedge Y_k.
\]
\item  Notice that $\Sigma (X_1\times X_2)\simeq \Sigma X_1\vee \Sigma X_2\vee \Sigma (X_1\wedge X_2)$. Now by statements $(1)$ and $(2)$, $\Sigma^n (X_1\times X_2)\simeq\Sigma^n X_1\vee\Sigma^n X_2\vee\Sigma^n (X_1\wedge X_2)\simeq\Sigma^{n+1}Y_1\vee\Sigma^{n+1}Y_2\vee\Sigma^{n+2}(Y_1\wedge Y_2) $. The proof of statement $(3)$ now follows by induction on $k$.
\end{enumerate}
\end{proof}

A generalisation of Lemma~\ref{wedge} can be given in terms of the Whitehead filtration of
$X_1\times\ldots\times X_k$.
For $k$ pointed topological spaces $(X_i, *)$, the \emph{Whitehead filtration} $T^k_{k-l}$ is defined in the following way
\[
T^k_{k-l}=\{(x_1,\ldots. x_k) \ |\ \text{ where $x_i=*$ for at least $l$ coordinates}\}.
\]
Notice that $T^k_k=X_1\times\ldots\times X_k$, $T^k_{k-1}$ is known as the fat wedge, $T^k_{1}=X_1\vee\ldots\vee X_k$, and $T^k_0=*$.
\begin{lem}
\label{whitehead}
Let $X_1,\ldots,X_k$ be $S^nS$-spaces. Then $T^k_l$ is $S^nS$-space for ${0\leq l\leq k}$.
\end{lem}
\begin{proof}
We prove the lemma by induction on $l$. The statement for  $l=1$ is true by Lemma~\ref{wedge}(1). Let us assume that $T^k_l$ is an $S^nS$-space. Notice that $T^k_l$ includes in $T^k_{l+1}$ for $0\leq l\leq k-1$ and that there is a cofibration sequence
\[
T^k_l\lra T^k_{l+1}\lra \bigvee_{1\leq i_1<\ldots<i_{l+1} \leq k}X_{i_1}\wedge X_{i_2}\wedge\ldots\wedge X_{i_{l+1}}
\]
which after being suspended becomes trivial, that is,
\[
\Sigma T^k_{l+1}\simeq \Sigma T^k_l\bigvee \left( \bigvee_{1\leq i_1<\ldots<i_{l+1} \leq k}\Sigma X_{i_1}\wedge X_{i_2}\wedge\ldots\wedge X_{i_{l+1}}\right).
\]
Now by the induction hypothesis and Lemma~\ref{wedge}(2), we get that $T^k_{l+1}$ is an $S^nS$-space, which proves the lemma.
\end{proof}
\begin{lem}
Let $X$ be an $S^nS$-space. Then
\begin{enumerate}
\item the space $\Omega\Sigma^{n} X$ is  an $S^1S$-space,
\item the space $\Omega\Sigma X$ is an $S^nS$-space.
\end{enumerate}
\end{lem}
\begin{proof}
Both statements are direct corollaries of  the James splitting~\cite{Ja},  $\Sigma\Omega\Sigma Z\simeq \bigvee_{k=1}^{\infty} \Sigma Z^{(k)}$, where $Z^{(k)}$ denotes the $k$-fold smash power of $Z$.
\end{proof}

The above elementary operations over  $S^nS$-spaces already indicate that the class of $S^nS$-spaces is wide.

In the remainder of the paper, of special interest to us will be the case $n=1$. We show that the class of $S^1S$-spaces is unexpectedly large by describing non-trivial homotopy theoretic constructions which produce $S^1S$-spaces.

\section{Applications of  the theory of Hopf algebras to $S^1S$-spaces}
We start by recalling some fundamental definition of the theory of Hopf algebras.
Let $R$ be a commutative ring.
\begin{defin}
An $R$-algebra $(A, m_A, \eta_A)$ is a \emph{Hopf algebra} if it has an additional structure given by $R$-algebra homomorphisms: $\Delta_A\colon A\lra A\otimes_RA$ called comultiplication; $\epsilon_A\colon A\lra R$ called counit, and an $R$-module homomorphism $S_A\colon A\lra A$ called antipode that satisfy the following properties
\begin{enumerate}
\item coassociativity:
\[
(\Id_A\otimes\Delta_A)\Delta_A=(\Delta_A\otimes\Id_A)\Delta\colon A\lra A\otimes A\otimes A;
\]
\item counitarity:
\[
m_A(\Id_A\otimes\epsilon_A)\Delta_A=\Id_A=m_A(\epsilon_A\otimes \Id_A)\Delta_A;
\]
\item antipode property:
\[
m_A(\Id_A\otimes S_A)\Delta_A=\eta_A\epsilon_A=m_A(S_A\otimes \Id_A)\Delta
\]
where $m_A\colon A\otimes A\lra A$ is the multiplication in $A$ and $\eta_A$ is the unit map.
\end{enumerate}
\end{defin}
\begin{defin}
Two Hopf algebras  $(A, m_A, \Delta_A, \epsilon_A, \eta_A, S_A)$ and $(B, m_B, \Delta_B, \epsilon_B, \eta_B, S_B)$ over $R$ are \emph{isomorphic}  if there is an algebra isomorphism $f\colon (A, m_A, \eta_A)\lra (B, m_B,\eta_B)$ satisfying
\begin{enumerate}
\item[i)] $(f\otimes f)\circ\Delta_A=\Delta_B\circ f$,
\item[ii)] $f\circ S_A=S_B\circ f$,
\item[iii)] $\epsilon_A=\epsilon_B\circ f$.
\end{enumerate}
\end{defin}
Note that $f$ uniquely determines the Hopf algebra structure of $B$ if the Hopf algebra structure of $A$ is given.

\begin{defin}
A Hopf algebra $A$ is called a \emph{Lie-Hopf} algebra if the set of multiplicative generators $\{ a_i\}_{i\in\mathbb{N}}$ comprises primitives, that is,
\[
\Delta a_i=a_i\otimes 1+1\otimes a_i\ \text{ for every $i$}.
\]
\end{defin}
Let $R$ be a PID and let $X$ be a connected space such that $H_*(X; R)$ is torsion free. Then
the  Bott-Samelson theorem~\cite{BS} asserts that the homology $H_*(\Omega\Sigma X; R)$ is isomorphic as an algebra to the tensor algebra $T(\widetilde{H}_*(X; R))$ and the adjoint of the identity on $\Sigma X$, the suspension map $E\colon X\lra\Omega \Sigma X$ induces the canonical inclusion of $\widetilde{H}_*(X;R)$ into $T(\widetilde{H}_*(X; R))$. This tensor algebra can be given many different Hopf algebra structures by considering different coalgebra structures on the module $\widetilde{H}_*(X; R)$ not necessarily coming form the topology of the space $X$. In the following proposition we identify  that Hopf algebra structure that is naturally coming from the topology of the space $\Omega\Sigma X$. 

\begin{proposition}
\label{generatingcomultiplication}
Let  $T(\widetilde{H}_*(X;R))$ be the Hopf algebra where the coalgebra structure is generated by the comultiplication  on  $\widetilde{H}_*(X;R)$ induced by the diagonal $\Delta_X\colon X\lra X\times X$. This Hopf algebra structure  coincides with the Hopf algebra structure of $H_*(\Omega\Sigma X;R)$ where the comultiplication is induced by the diagonal map $\Delta_{\Omega\Sigma X}\colon\Omega\Sigma X\lra\Omega\Sigma X\times\Omega\Sigma X$.
\end{proposition}
\begin{proof}
For any $H$-space $H$ with multiplication $\mu_H$, we have the following commutative diagram
\[
\xymatrix{
H\times H \ar[rrr]^{\mu_H} \ar[d]^{\Delta_{H}\times \Delta_{H}} && &H\ar[d]^{\Delta_{H}}\\
H\times H\times H\times H \ar[rr]^{\Id\times \mathrm{T}\times\Id} && H\times H\times H\times H\ar[r]^-{\mu_{H}\times\mu_{H}} &H\times H}
\]
where T is the twist map. Since the composite $(\mu_H\times\mu_H)\circ (\Id\times \mathrm{T}\times\Id)$ is the natural multiplication on $H\times H$ induced by $\mu_H$, we have that the diagonal $\Delta_H$ is  a multiplicative map.

To prove the proposition notice further that $\Omega\Sigma X$ is a universal space in the category of homotopy associative $H$-spaces (see~\cite{Ja}). The universal property states that any $H$-map from $\Omega\Sigma X$ to a homotopy associative space is determined by its restriction to $X$.
Now since the diagonal $\Delta_{\Omega\Sigma X}\colon \Omega\Sigma X\lra\Omega\Sigma X\times\Omega\Sigma X$ is an $H$-map, it is determined by its restriction to $X$ which is the composite $\ddr{X}{\Delta_X}{X\times X}{E\times E}{ \Omega\Sigma X\times\Omega\Sigma X}$. This proves the proposition.
\end{proof}

\begin{cor}
\label{coHcorollary}
Let $C$ be a co-$H$-space. Then the Hopf algebra $H_*(\Omega\Sigma C;R)$ is a Lie-Hopf algebra, that is, it is primitively generated.
\end{cor}
\begin{proof}
By Proposition~\ref{generatingcomultiplication}, the comultiplication in  $H_*(\Omega\Sigma C;R)$ is generated by the comultiplication on $\widetilde{H}_*(C;R)$. Now directly from the definition of a co-H-space, we have that the reduced diagonal $\widetilde{\Delta}\colon C\lra C\wedge C$ is trivial, showing that $\widetilde{H}_*(C;R)$ consists only of primitive elements.
\end{proof}

We generalise the last statement to spaces which are not necessarily co-$H$-spaces but once suspended become a suspension of a co-$H$-space.

\begin{thm}
\label{coH}
Let $X$ be a topological space such that $\Sigma X\simeq \Sigma C$ where $C$ is a co-$H$-space and assume that $H_*(X;\mathbb{Z})$ is torsion free. Then over $\mathbb{Z}$ the Hopf algebra $H_*(\Omega\Sigma X;\coefZ)$  is isomorphic to a Lie-Hopf algebra, that is, to a primitively generated Hopf algebra.
\end{thm}
\begin{proof}
Let $\varphi\colon \Sigma X\lra\Sigma C$ be a homotopy equivalence. Since  $\Omega \varphi$ is an $H$-map it induces a Hopf algebra morphism $\Omega\varphi_*\colon H_*(\Omega\Sigma X)\lra H_*(\Omega\Sigma C)$ which is an isomorphism of Hopf algebras as $\Omega \varphi$ is a homotopy equivalence. Now the statement of the theorem follows from Corollary~\ref{coHcorollary} since $H_*(\Omega\Sigma C)$ is a Lie-Hopf algebra.
\end{proof}
Throughout the rest of this paper we assume the coefficient ring $R=\coefZ$ and thus, for the sake of convenience, in integral homology and cohomology we suppress the coefficients from the notation.

Let $\varphi \colon \Sigma X \to \Sigma^2 Y$ be a homotopy equivalence. Then by the Bott-Samelson theorem, we have
\[
H_*(\Omega\Sigma X)\cong T(\widetilde H_*(X)), \qquad
H_*(\Omega\Sigma^2 Y)\cong T(\widetilde H_*(\Sigma Y)).
\]
Let  $\{a_i\}$ be an additive basis for $\widetilde H_*(X)$ and let $\{ b_i\}$ be an additive basis for $\widetilde H_*(\Sigma Y)$. Then the elements $b_i$ are primitive, that is,
$\Delta b_i = 1\otimes b_i+b_i\otimes 1$. If we know the cohomology  ring $H^*(X)$, we can calculate the comultiplication $\Delta_X \colon H_*(X) \to H_*(X)\otimes H_*(X)$.

\begin{lem}
The comultiplication $\Delta_X \colon H_*(X) \to H_*(X)\otimes H_*(X)$
is determined by the Hopf algebra homomorphism
\[ \varphi_* \colon H_*(\Omega\Sigma X) \longrightarrow H_*(\Omega\Sigma^2 Y). \]
\end{lem}
\begin{proof}
From the definition of a Hopf algebra isomorphism, the comultiplication $\Delta_X$ is determined by the formula
$\Delta_{\Omega\Sigma^2
Y}\varphi_*(a) = \varphi_*\otimes\varphi_*(\Delta_{\Omega\Sigma
X}a)$ where $a\in H_*(X)$.
\end{proof}
\begin{proposition}
The Hopf algebra $H_*(\Omega\Sigma X)$ for $X=\Omega \Sigma Z$ with $Z$ a co-$H$-space is isomorphic to a Lie-Hopf algebra. The change of homology generators is given by Hopf invariants.
\end{proposition}
\begin{proof} The proof follows from the James splitting of $ X$, that is, $\Sigma X\simeq\Sigma\Omega\Sigma Z\simeq\Sigma(\bigvee_{k=1}^\infty Z^{(k)})$. Since $Z$ is a co-H-space, $H_*(\Omega\Sigma(\bigvee_{k=1}^\infty Z^{(k)}))$ is primitively generated. Since the James splitting is given by Hopf invariants, the statement of the proposition follows.
\end{proof}

Hopf and Borel used the theory of  Hopf algebras to find a necessary condition on homology of $X$ so that $X$ is an $H$-space. Using the same theory but studying in detail the coalgebra structure, we give a necessary condition on the homology of $X$ to be an $S^1S$-space.
\begin{cor}
\label{s1s}
Let $X$ be a CW-complex such that $H_*(X)$ is torsion free and let the diagonal  $\Delta \colon
H_*(X) \to H_*(X)\otimes H_*(X)$  define a Hopf algebra structure on $H_*(\Omega\Sigma X)\cong T(\widetilde H_*(X))$. If the Hopf algebra $H_*(\Omega\Sigma X)$ is not isomorphic to a Lie-Hopf algebra, then $X$ is not an $S^1S$-space.\qed
\end{cor}

\section{Examples of $S^1S$-spaces}
\label{examples}

In this section we present several examples of $S^1S$-spaces. In doing so we also illustrate how our invariant can be explicitly calculated and related to some classical homotopy invariants.
\subsection{``Small'' CW-complexes}
\label{small}
The ``smallest'' topological spaces amongst which we can find non-trivial $S^1S$-spaces are two cell complexes. Let $\alpha\colon S^k\lra S^l$ where $k>l$ such that $\Sigma\alpha\simeq *$. Taking its homotopy cofibre, we obtain a space $X=S^l\cup_\alpha e^{k+1}$ which is an $S^1S$-space. Slightly more generally start with a wedge of spheres and attach to them a single cell by an  attaching map $\alpha$ such that $\Sigma\alpha\simeq *$. In this case we also obtain an $S^1S$-space. As an explicit example, take  the Whitehead product $\dr{S^{m+n-1}}{\omega}{S^m\vee S^n} $ which is the attaching map for $S^m\times S^n$. As $\Sigma\omega\simeq *$, there is a stable splitting $\Sigma(S^m\times S^n)\simeq S^{m+1}\vee S^{n+1}\vee S^{m+n+1}$, proving that $S^m\times S^n$ is an $S^1S$-space.

\subsection{Simply-connected $4$-manifolds}

Note that the classical Pontryagin-Whitehead theorem (see for example~\cite{Nov}) asserts that closed simply-connected 4-dimensional manifolds with isomorphic integer cohomology rings are homotopy equivalent.
This result was strengthen by Freedman who proved that smooth, closed simply-connected 4-manifolds with isomorphic integer cohomology rings are homeomorphic.

Using our invariant we can reformulate Pontryagin-Whitehead and Freedman theorem in the following way.

\begin{proposition}
Let $M_1$ and $M_2$ be closed simply-connected 4-manifolds. If $\Sigma M_1\simeq \Sigma M_2$, then the manifolds $M_1$ and $M_2$ are homotopy equivalent.
\end{proposition}
\begin{proof}
The statement will follow readily from Proposition~\ref{mnflds}.
\end{proof}
\begin{cor}
Let $M_1$ and $M_2$ be smooth closed simply-connected 4-manifolds. If $\Sigma M_1\simeq \Sigma M_2$, then $M_1$ and $M_2$ are homeomorphic. \qed
\end{cor}

Recall that the Hopf invariant is a classical obstruction to $\Sigma \mathbb{C}P^2$ being a double suspension. Here, as a mere illustration of the calculation of our invariant, we show that that $\C P^2$ is not an $S^1S$-space. To prove that,  according to Corollary~\ref{s1s} we need to show that the Hopf algebra $H_*(\Omega\Sigma\C P^2)$ is not isomorphic to a Lie-Hopf algebra. We first calculate the coproduct structure of  $H_*(\C P^2)$.  Let
$u_1\in H_2(\C P^2)$ and $u_2\in H_4(\C P^2)$ be generators. Then the coproduct in $H_*(\C P^2)$ is determined by
$\Delta u_1=1\otimes u_1+u_1\otimes 1$ and $\Delta u_2=1\otimes u_2+u_1\otimes u_1 +u_2\otimes 1$.
By Proposition~\ref{generatingcomultiplication}, the coproduct in the Hopf algebra $H_*(\Omega\Sigma\mathbb{C}P^2)\cong T(u_1,u_2)$ is determined by the coproduct in $H_*(\C P^2)$. Thus we have
\[
\Delta u_1=1\otimes u_1+u_1\otimes 1,\quad \Delta u_2=1\otimes u_2+
u_1\otimes u_1 +u_2\otimes 1.
\]
We use the bar notation to denote the tensor product in $T(u_1, u_2)$.
The group $H_4(\Omega\Sigma\C P^2)=\mathbb{Z}\oplus\mathbb{Z}$ is generated by  $u_2$ and $u_1|u_1$. The only possible change of basis is given by $w_1=u_1,\; w_2=u_2+\lambda u_1|u_1$ for some integer $\lambda$. From the fact that $w_2$ is a primitive element, we have the following relation
\[
\Delta w_2=1\otimes w_2+w_2\otimes 1 = 1\otimes u_2+u_1\otimes u_1 +u_2\otimes 1 + \lambda(1\otimes u_1+u_1\otimes 1)|(1\otimes u_1+u_1\otimes 1).
\]
Thus we obtain the following condition on $\lambda$
\begin{equation}
\label{lambda}
1+2\lambda=0.
 \end{equation}
This equation has no solution over the integers, proving that $\C P^2$ is not an $S^1S$-space as there is no Hopf algebra isomorphism between $H_*(\Omega\Sigma\C P^2)$ and a Lie-Hopf algebra. In this way we have also reproved the classical well known result that the suspension of the Hopf map $S^3\lra S^2$ is not null homotopic.

\begin{rem}
In an analogous way we can prove that no Hopf invariant one complexes are $S^1S$-spaces and as a consequence we see that the suspension of the Hopf map is not null homotopic.
\end{rem}

If instead of $\mathbb{C}P^2$ we take  $S^2\times S^2$, then the analogous equation to \eqref{lambda} takes the form  $1+\lambda=0$ which is solvable over $\coefZ$ by taking $\lambda=-1$. Thus there is a change of basis in $H_*(\Omega\Sigma(S^2\times S^2))$ which induces a Hopf algebra isomorphism with a Lie-Hopf algebra. Topologically, following the arguments of Subsection~\ref{small},  we also know that  $S^2\times S^2$ is an $S^1S$-space, as $\Sigma (S^2\times S^2)\simeq\Sigma^2( S^1\vee S^1\vee S^3)$.

To approach simply-connected $4$-dimensional manifolds $M^4$ more generally, we use their classifications in terms of the intersection form on $H^2(M^4)$. Let us denote the intersection form matrix by $A$. We ask for a condition on $A$ so that $M^4$ is an $S^1S$-space. Following our programme we want to construct a looped homotopy equivalence $\Omega\Sigma M^4\lra \Omega\Sigma^2 Y$ for some $CW$-complex $Y$.
Let us assume that $H_2(M^4)\cong \coefZ^k$ is generated by $u_i$ for $1\leq i\leq k$. For dimensional reasons all $u_i$ are primitive. If the matrix $A$ is non-trivial, then a generator $v\in H_4(M^4)$ is not primitive and its coproduct is given by $\Delta v=v\otimes 1+1\otimes v+u^\top A u$.
 Thus we are looking for  a change of basis
\[
\dr{u_i}{f}\lambda_{i1}w_1+\lambda_{i2}w_2+\ldots +\lambda_{ik}w_k
\]
where all $w_i$ for $1\leq i\leq k$ are primitive.
In matrix form $f(u)=\Lambda w$ where $\Lambda$ is the $k$-matrix $(\lambda_{ij})$ such that $\det \Lambda =\pm 1$.
In degree 4, we look to find a primitive generator $\widetilde{w}\in H_4(M^4)$. For dimensional reasons, $f(v)= \widetilde{w}+\sum \gamma_{kl}w_k|w_l$ for some $\gamma_{kl}\in\coefZ$; or equivalently, in matrix form $f(v)=\widetilde{w}+w^\top\Gamma w$ for the $k$-matrix $\Gamma=(\gamma_{ij})$.

For $M^4$ to be an $S^1S$-space, the following diagram needs to commute
\[
\xymatrix{
v\ar[d]_{\Delta}\ar[r]^f & \widetilde{w}+w^{\top}\Gamma w\ar[d]^{\Delta}\\
v\otimes 1+1\otimes v+u^{\top}Au\ar[r]^-f & \widetilde{w}\otimes 1+1\otimes \widetilde{w} +(w\otimes 1+1\otimes w)^\top\Gamma(w\otimes 1+1\otimes w).
}
\]
From here we deduce the necessary condition
\[
\Lambda^\top A\Lambda=-\Gamma-\Gamma^\top
\]
or equivalently, if we denote $\Lambda^{-1}$ by L, then
\[
A=L^\top(-\Gamma-\Gamma^\top)L.
\]
\begin*{rem}
Since all $u_i$ are  primitive, as a particular case we can choose $\Lambda$ to be the identity matrix.
\end*{rem}
For an easy illustrative example, consider $S^2\times S^2$ and its associated quadratic matrix
\[
A=\begin{pmatrix}0 & 1 \\ 1 & 0\end{pmatrix}.
\]
Taking $\Lambda$ to be $\begin{pmatrix}1 & 0 \\ 0 & 1\end{pmatrix}$, then $\Gamma=\begin{pmatrix}0 & -1\\  0 & 0\end{pmatrix}$ satisfies the condition $A=L^\top(-\Gamma-\Gamma^\top)L$.

\begin{proposition}
\label{mnflds}
Let $M_1$ and $M_2$ be simply-connected $4$-dimensional manifolds. Then $M_1$ and $M_2$ are homotopy equivalent if and only if as Hopf algebras, $H_*(\Omega\Sigma M_1)$ and $H_*(\Omega\Sigma M_2)$ are isomorphic over $\coefZ$.
\end{proposition}
\begin{proof}
Recall that the homotopy type of simply-connected $4$-dimensional manifold $M$ is classified in terms of the intersection form on $H^2(M)$. A Hopf algebra isomorphism of $H_*(\Omega\Sigma M_1)\cong T(\widetilde H_*(M_1))$ and $H_*(\Omega\Sigma M_2)\cong T(\widetilde H_*(M_2))$
induces an isomorphism of the intersection forms on $H^2(M_1)$ and $H^2(M_2)$, which proves the non-trivial part of the proposition. 
\end{proof}

\subsection{The complement of a hyperplane arrangement}

Let $\A$ be a complex hyperplane arrangement  in $\C^l$, that is, a finite set of hyperplanes in $\C^l$. Denote by  $M(\A)$ its complement, that is, $M(\A)= \C^l\setminus\hspace{-.1cm} \text{ supp } \A$. The cohomology of $M(\A)$ is given by the Orlik-Solomon algebra $A(\A)$ (see for example~\cite{OT}). As there are many non-trivial products in this algebra, we can see the Hopf algebra $H_*(\Omega\Sigma M(\A))$ is not primitively generated. On the other hand, it is well known that $\Sigma M(\A)$ breaks into a wedge of spheres (see for example~\cite{S}) and therefore it is an $S^1S$-space. Thus we can conclude that the Hopf algebra $H_*(\Omega\Sigma M(\A))$ is isomorphic to the Lie-Hopf algebra $H_*(\Omega\bigvee S^{n_\alpha})$. In a subsequent paper we will study properties of the cup product in $H^*(M(\A))$, that is, of the Orlik-Solomon algebra $A(\A)$.

\section{Polyhedral products as $S^1S$-spaces}
\label{polyhedral}

A new large family of $S^1S$-spaces appeared as a result of recent work of Bahri, Bendersky, Cohen, and Gitler~\cite{BBCG}. We start by recalling the definition of a polyhedral product functor. Let $K$ be an abstract simplicial complex on $m$ vertices, that is, a finite set of subsets of $[m]=\{ 1,\ldots, m\}$ which is closed under formation of subsets and includes the empty set. Let $(\ul{X},\ul{A})=\{ (X_i, A_i, x_i)\}^m_{i=1}$ denote $m$ choices of connected, pointed pairs of CW-complexes. Define the functor $D\colon K\lra CW_*$ by
\[
D(\sigma)=\prod_{i=1}^m B_i, \text{ where } B_i=\left\{
\begin{array}{ccl}
X_i  & \text{ if }  &  i\in\sigma  \\
A_i  &  \text{ if } & i\in [m]\setminus \sigma
\end{array}
\right .
\]
with $D(\emptyset)=A_1\times\ldots\times A_m$.
Then the polyhedral product $(\ul{X}, \ul{A})^K$ is given by
\[
(\ul{X}, \ul{A})^K=\colim_{\sigma\in K} D(\sigma).
\]
When $(\ul{X},\ul{A})=(D^2, S^1)$, we recover the definition of the moment-angle complex $\mathcal{Z}_K$ introduced by Buchstaber and Panov~\cite{BP}.

In~\cite{BBCG} it was proved that
\[
\Sigma (D^{n+1},S^n)^K\simeq \bigvee_{I\notin K}\Sigma^{2+n|I|}|K_I|
\]
which shows that $(D^{n+1},S^n)^K$, and in particular the moment-angle complexes $\mathcal{Z}_K$, are $S^1S$-spaces. This shows that the collection of $S^1S$-spaces is large and consists of many important spaces which are studied in various mathematical disciplines such as toric topology, complex, symplectic and algebraic geometry, combinatoric and so on.

A natural question that arises is to determine which polyhedral products are $S^1S$-spaces.
Recall that for two topological spaces $X$ and $Y$, the join $X\ast Y$ is homotopy equivalent to
$\Sigma X\wedge Y$. Following~\cite{BBCG}, the space $\widehat{D}(\sigma)$ denotes the smash product $B_{i_1}\wedge\ldots\wedge B_{i_k}$ for $\sigma=\{i_1,\ldots, i_k\}$.  Similarly, $\widehat X^I$ denotes the smash product $X_{i_1}\wedge\ldots\wedge X_{i_k}$ for $I=\{i_1,\ldots, i_k\}$. We now collect three statements on the stable homotopy type of certain polyhedral product functors proven in~\cite{BBCG}.

\begin{thm}[~\cite{BBCG}]
\label{iAcont}
Let $K$ be an abstract simplicial complex with $m$ vertices, and let
\[(\ul{X},\ul{A})=\{ (X_i, A_i, x_i)\}^m_{i=1}
\]
denote $m$ choices of connected, pointed pairs of CW-complexes with the inclusion $A_i\subset X_i$ null-homotopic for all $i$. Then there is a homotopy equivalence
\[
\Sigma (X,A)^K\lra\Sigma\left(\bigvee_I\left(\bigvee_{\sigma\in K_I}|\Delta(\bar{K}_I)_{<\sigma}|\ast \widehat{D}(\sigma)\right)\right).
 \]\qed
 \end{thm}

 \begin{thm}[~\cite{BBCG}]
 \label{Acont}
If all of the $A_i$ are contractible with $X_i$ and $A_i$ closed CW-complexes for all $i$, then there is a homotopy equivalence
 \[
\Sigma ( X, A)^K \lra\Sigma\left(\bigvee_{I\in K}\widehat{X}^I\right).
\]\qed
\end{thm}

\begin{thm}[~\cite{BBCG}]
\label{Xcont}
If all of the $X_i$ in $(\ul{X}, \ul{A})$ are contractible with $X_i$ and $A_i$ closed CW-complexes for all $i$, then there is a homotopy equivalence
\[
\Sigma (X, A)^K\lra\Sigma\left(\bigvee_{I\notin K}| K_I |\ast \widehat{A}^I \right).
\]\qed
\end{thm}

\begin{cor}
The polyhedral products considered in Theorems~\ref{iAcont} and~\ref{Xcont} are $S^1S$-spaces. For the polyhedral products in Theorem~\ref{Acont}  to be $S^1S$-spaces we additionally need each $X_i$ to be a suspension space.\qed
\end{cor}

\subsection{Structural properties of the cup product}
From the point of view of cohomology rings,  a particularly interesting case of polyhedral products is the moment-angle complex $\Z_K=(D^2, S^1)^K$ where $K$ is an arbitrary simplicial complex. The cohomology of the moment-angle complex is known in terms of the Stanley-Reisner algebra $\coefZ [K]$ (see for example~\cite{BP}). Our approach to the theory of Hopf algebras allows us to give a condition on the algebra necessary for it to be realised as the cohomology of a moment-angle complex.  The previous analysis showed that the moment-angle complexes are $S^1S$-spaces, implying that the Hopf algebra $H_*(\Omega\Sigma \Z_K;\coefZ)$ is isomorphic, as a Hopf algebra, to a Lie-Hopf algebra. In other words, in the Hopf algebra $T(\widetilde H_*(\Z_K))$ where the coproduct is induced by the coproduct in $H_*(\Z_K)$ there is a change of basis such that $T( \widetilde H_*(\Z_K))$ becomes primitively generated. Dually, this property can be seen as a new structural property of the cup product. 
Thus we proved the following characterisation property of the cup product for the moment-angle complexes.
 
\begin{proposition}
For an algebra $A$ to be the cohomology algebra of the moment-angle complex $\Z_K$, there must exist a change of basis in $T(A^*)$ such that $T(A^*)$ becomes a Lie-Hopf algebra, where $A^*$ denotes the dual of $A$.\qed
\end{proposition}

This statement readily generalises to the whole family of $S^1S$-spaces.
\begin{proposition}
For an algebra $A$ to be the cohomology algebra of an $S^1S$-space, there must exist a change of basis in $T(A^*)$ such that $T(A^*)$ becomes a Lie-Hopf algebra. \qed
\end{proposition} 

\begin{exa}
Let us consider the 6-dimensional moment-angle manifold $M=(D^2,S^1)^K$ where $K$ is a square, that is,
\[
K=\big\{\{1\}, \{2\}, \{3\},\{4\},\{1,2\}, \{2,3\}, \{3,4\}, \{1,4\}\big\}.
\]
From the work of Buchstaber and Panov~\cite{BP}, we know
\[
H^*((D^2,S^1)^K;\coefZ)\cong H^*\left[ \Lambda [u_1,u_2,u_3,u_4]\otimes\coefZ[v_1,v_2,v_3,v_4]/(v_1v_3, v_2v_4) \ ; \ d\right]
\]
where $|u_1|=1, |v_i|=2$ and the differential $d$ is given by $d(u_i)=v_i$ and $d(v_i)=0$ for $1\leq i\leq 4$. By a straightforward calculation, we find that the 3-dimensional cycles are $a_1=u_1v_3,\ a_2=u_2v_4$, the 6-dimensional cycle is $b=u_1u_2v_3v_4$, and the intersection form on $H^3(M^6)$ is given by the following matrix
\[
A=\begin{pmatrix}0 & 1 \\ -1 & 0\end{pmatrix}.
\]
Thus in integral homology we have $H_*(\Omega\Sigma M ;\coefZ)\cong T(a_1, a_2, b)$, where $|a_1|=|a_2|=3$ and $|b|=6$, and the comultiplication is given by $\Delta(a_i)=1\otimes a_i +a_i\otimes 1$ for $i=1,2$ and  $\Delta(b)=1\otimes b + a_1\otimes a_2 -a_2\otimes a_1 +b\otimes 1$. Therefore this Hopf algebra is not primitively generated. We want to show that this Hopf algebra is however isomorphic to a Lie-Hopf algebra by describing  the change of this basis to a primitive one.
As $a_1$ and $a_2$ are primitive elements, we take $w_i=a_i$ for $i=1,2$. For dimensional reasons, $w_3= b+ \lambda_1 a_1|a_2 +\lambda_2 a_2|a_1+\lambda_3 a_1|a_1+\lambda_4 a_2|a_2$. As we want $w_3$ to be primitive, we need the following relation to hold
\[
\Delta(w_3)=1\otimes w_3+w_3\otimes 1= \Delta (b)+\lambda_1 \Delta(a_1)\Delta(a_2) +\lambda_2\Delta(a_2)\Delta(a_1)+\lambda_3\Delta(a_1)\Delta(a_1)+\lambda_4\Delta(a_2)\Delta(a_2).
\]
By a direct calculation, we conclude that for $w_3$ one can take, for example, $w_3=b+a_2|a_1$. Thus we have confirmed that $H_*(\Omega\Sigma M;\coefZ)$ is isomorphic to a Lie-Hopf algebra.
\end{exa}
\begin{exa}
Now let us look at  the 7-dimensional moment-angle manifold $M=(D^2,S^1)^K$ where $K$ is a pentagon, that is,
\[
K=\big\{\{1\}, \{2\}, \{3\},\{4\},\{5 \},\{1,2\}, \{2,3\}, \{3,4\},\{ 4,5\} \{1,5\}\big\}.
\]
Following~\cite{BP}, the cohomology algebra is given ny
\[
H^*((D^2,S^1)^K;\coefZ)\cong H^*\left[ \Lambda [u_1,u_2,u_3,u_4, u_5]\otimes\coefZ[v_1,v_2,v_3,v_4, v_5]/(v_1v_3, v_1v_4,v_2v_4, v_2v_5,v_3v_5) \ ; \ d\right]
\]
where $|u_1|=1, |v_i|=2$ and the differential $d$ is given by $d(u_i)=v_i$ and $d(v_i)=0$ for $1\leq i\leq 5$. There are five 3-dimensional cycles $a_1=u_1v_3, a_2=u_4v_1, a_3=u_2v_4, a_4=u_5v_2, a_5=u_3v_5$; five 4-dimensional cycles $b_1=u_4u_5v_2, b_2=u_2u_3v_5, b_3=u_5u_1v_3, b_4=u_3u_4v_1, b_5=u_1u_2v_4$ and one 7-dimensional cycle $c=u_1u_2u_3v_4v_5$. All non-trivial cup products are given by $a_ib_i=b_ia_i=c$ for $1\leq i\leq5$.
Thus in integral homology we have $H_*(\Omega\Sigma M ;\coefZ)\cong T(\{a_i, b_i\}_{i=1}^5, c)$, where $|a_i|=3$, $|b_i|=4$ and $|c|=7$, and the comultiplication is given by $\Delta(a_i)=1\otimes a_i +a_i\otimes 1$, $\Delta(b_i)=1\otimes b_i +b_i\otimes 1$  for $1\leq i\leq5$ and  $\Delta(c)=1\otimes c +\sum_{i=1}^5 (a_i\otimes b_i + b_i\otimes a_i) +c\otimes 1$. Thus this Hopf algebra is not primitively generated. We want to show that this Hopf algebra is however isomorphic to a Lie-Hopf algebra by describing  the change of this basis to a primitive one.
As $a_i$ and $b_i$ are primitive elements, we take for $w_{2i-1}=a_i$ and $w_{2i}=b_i$ for $1\leq i\leq 5$. For dimensional reasons, $w_{11}= c+ \sum_{i=1}^5(\lambda_i a_i|b_i +\lambda_{i+5} b_i|a_i)$. As we want $w_{11}$ to be primitive, we need the following relation to hold
\[
\Delta(w_{11})=1\otimes w_{11}+w_{11}\otimes 1= \Delta (c)+\sum_{i=1}^5\left(\lambda_i \Delta(a_i)\Delta(b_i) +\lambda_{i+5}\Delta(b_i)\Delta(a_i)\right).
\]
We conclude that for $w_{11}$ one can take, for example, $w_{11}=c-\sum_{i=1}^5a_i|b_i$. Thus we have confirmed that $H_*(\Omega\Sigma M;\coefZ)$ is isomorphic to a Lie-Hopf algebra.
\end{exa}
\section{Applications to algebra}
\label{applications}
\subsection{Quasi-symmetric polynomials}
In this section we use our new homotopy invariant to study the ring of quasi-symmetric functions. We start by recalling the main definitions, following mainly the notation of Hazewinkel~\cite{Ha}. For more details on applications of topological methods to the study of quasi-symmetric functions see Buchstaber and Erokhovets~\cite{BE}.

\begin{defin}
A \emph{composition} $\omega$ of a number $n$ is an ordered set $\omega = (j_1,\ldots, j_k)$, for $j_i > 1$, such that $n = j_1 + \ldots+ j_k$. Let us denote $|\omega| = n$, $l(\omega) = k$.
The empty composition of $0$ we denote by $()$. Then $|()| = 0,\  l (()) = 0$.
\end{defin}
\begin{defin}
Let $t_1, t_2,\ldots$ be a finite or an infinite set of variables of degree 2. For a composition $\omega= (j_1,\ldots, j_k)$, consider a quasi-symmetric monomial
\[
M_\omega=\sum_{ l_1<\ldots<l_k} t^{j_1}_{l_1}\ldots t^{j_k}_{l_k}, \quad  M_{()} = 1.
\]
whose degree is equal to $2|\omega| = 2(j_1 + \ldots + j_k)$.

For any two monomials $M_{\omega'}$ and $M_{\omega''}$, their product in the ring of polynomials $\coefZ[t_1, t_2,\ldots]$ is equal to
\[
M_{\omega'}M_{\omega''} =\sum_{\omega}\left(\sum_{\Omega'+\Omega''=\omega} 1\right)
M_{\omega}
\]
where for the compositions $\omega= (j_1,\ldots, j_k),\  \omega' = (j_1',\ldots , j_{l'}' ),\  \omega'' = (j_1'' ,\ldots  , j_{l''}'' ), \Omega'$  and $\Omega''$ are all the $k$-tuples such that
\[
\Omega' = (0,\ldots, j_1',\ldots  , 0,\ldots, j_l',\ldots, 0), \ \Omega''= (0,\ldots, j_1'',\ldots, 0,\ldots, j_{l''}''\ldots, 0).
\]
This multiplication rule of compositions is called the \emph{overlapping shuffle multiplication}.
\end{defin}
Thus finite integer combinations of quasi-symmetric monomials form a ring. This ring is called the \emph{ring of quasi-symmetric functions} and is denoted by $\QSymm[t_1,\ldots, t_n]$, where $n$ is the number of variables. In the case of an infinite number of variables it is denoted by $\QSymm[t_1, t_2,\ldots ]$ or $\QSymm$.

The diagonal map $\Delta\colon\QSymm\lra\QSymm\otimes\QSymm$ given by
\[
\Delta M_{(a_1,\ldots, a_k)} =\sum^k_{i=0}M_{(a_1,\ldots, a_i)}\otimes M_{(a_{i+1},\ldots, a_k)}
\]
defines on $\QSymm$ the structure of a graded Hopf algebra.

In~\cite{Ha} Hazewinkel proved the Ditters conjecture that $\QSymm[t_1, t_2,\ldots]$ is a free commutative algebra of polynomials over the integers.

Let $R$ be a commutative associative ring with unit.
\begin{defin}
A \emph{Leibnitz-Hopf algebra} over the ring R is an associative Hopf algebra $\mathcal{H}$ over the ring $R$ with a fixed sequence of a finite or countable number of multiplicative generators $H_i$, $i=1,2,\ldots$ satisfying the comultiplication formula
\[
\Delta(H_n)=\sum_{i+j=n}H_i\otimes H_j, \quad H_0=1.
\]

A \emph{universal Leibnitz-Hopf algebra} $\A$ over the ring $R$ is a Leibnitz-Hopf algebra with the universal property: for any Leibnitz-Hopf algebra $\mathcal{H}$ over the ring $R$ the correspondence $A_i\lra H_i$ defines a Hopf algebra homomorphism.
\end{defin}

Of special interest to us will be the free associative Leibnitz-Hopf algebra over the integers $\mathcal{Z} =\coefZ\langle Z_1,Z_2,\ldots\rangle$ in countably many generators $Z_i$.

\begin{defin}
A \emph{universal commutative Leibnitz-Hopf algebra} $\mathcal{C} =\coefZ[C_1,C_2,\ldots ]$ is a free commutative polynomial Leibnitz-Hopf algebra in generators $C_i$ of degree $2i$. We have $\mathcal{C} =\mathcal{Z}/J_{\mathcal{C}}$, where the
ideal $J_{\mathcal{C}}$ is generated by the relations $Z_iZ_j-Z_jZ_i$.
\end{defin}
The Leibnitz-Hopf algebra $\mathcal{C}$ is a self-dual Hopf algebra and the graded dual Hopf algebra is naturally isomorphic to the algebra
of symmetric functions  $\coefZ[\sigma_1, \sigma_2, \ldots] = \Symm[t_1, t_2,\ldots ]\subset\QSymm[t_1, t_2,\ldots ]$ generated by the symmetric monomials
\[
\sigma_i = M_{\omega_i} = \sum_{l_1<\ldots<l_i}
t_{l_1}\ldots t_{l_i}
\]
where $\sigma_i = (\underbrace{1, . . . , 1}_i)$.

The isomorphism $\mathcal{C}=\mathcal{C}^*$ is given by the correspondence $C_i\lra\sigma_i$.

The Hopf algebra of symmetric functions $\Symm$ has a non-commutative analogue $\NSymm$ obtained by replacing the polynomial algebra in the definition by a free associative algebra $\NSymm=\coefZ\langle\sigma_1,\ldots,\sigma_m,\ldots\rangle$. The diagonal of $\NSymm$ is defined by the same formula as in $\Symm$ and is still cocommutative. The dual of $\NSymm$ is the commutative algebra of quasi-symmetric functions $\QSymm$.

Looking  at the homology of  $\Omega\Sigma\mathbf{C}P^\infty$ with integral coefficients we get a Hopf algebra which is isomorphic to $\NSymm$ (see~\cite{BR}).  This was the starting point for Baker and Richter to positively solve Ditters conjecture using topological methods. By calculating the cohomology algebra $H^*(\Omega\Sigma\C P^\infty)$ they showed that $\QSymm$ is a polynomial algebra.

For the sake of completeness and as an illustration of the close relation between topology and algebra, we explicitly calculate Hopf algebras related to the quasi-symmetric algebra $\QSymm$ by identifying them with the (co)homology of certain topological spaces.

We have the following Hopf algebras:

\textbf{I.}
\begin{enumerate}
\item $H_*(\C P^{\infty})$ is a divided power algebra $\mathbb Z[u_1,u_2,\ldots]/I$, where the ideal $I$ is generated by the relations $u_iu_j-{i+j\choose i}u_{i+j}$, with the comultiplication
\begin{equation}\label{F-1}
\Delta u_n=\sum_{k=0}^{n}u_k\otimes u_{n-k};
\end{equation}
\item $H^*(\C P^{\infty})\cong\mathbb Z[u]$ is the polynomial algebra generated by a single generator $u$ in degree 2.
\end{enumerate}

\textbf{II.}
\begin{enumerate}
\item $H_*(\Omega\Sigma\C P^{\infty})\cong T(\tilde H_*(\C P^{\infty}))=\mathbb Z\langle u_1,u_2,\ldots\rangle$ with $u_i$ being non-commuting variables of degree $2i$. Thus there is an isomorphism of rings
\[
H_*(\Omega\Sigma\C P^{\infty})\cong\mathcal{Z}
\]
under which $u_n$ corresponds to $Z_n$.

The coproduct $\Delta$ on $H_*(\Omega\Sigma\C P^{\infty})$ induced by the diagonal in $\Omega\Sigma\C P^{\infty}$ is compatible with the one in $\mathcal{Z}$:
\[
\Delta u_n=\sum\limits_{i+j=n}u_i\otimes u_j.
\]
Thus there is an isomorphism of graded Hopf algebras.

\item $H^*(\Omega\Sigma\C P^{\infty})$ is the graded dual Hopf algebra to $H_*(\Omega\Sigma\mathbb CP^{\infty})$.
\end{enumerate}

\textbf{III.}

$H_*(BU)\cong H^*(BU)\cong\mathbb Z[\sigma_1,\sigma_2,\dots]\cong\mathcal{C}$. It is a self-dual Hopf algebra of symmetric functions. In cohomology $\sigma_i$ are represented by Chern classes.

\textbf{IV.}
\begin{enumerate}
\item $H_*(\Omega\Sigma S^2)\cong H_*(\Omega S^3)\cong\mathbb Z[w]$ is a polynomial ring with $\deg w=2$ and the comultiplication
\[
\Delta w=1\otimes w+w\otimes 1
\]
\item $H^*(\Omega\Sigma S^2)\cong\mathbb Z[u_1,u_2,\ldots]/I$, where the ideal $I$ is generated by the relations $u_iu_j-{i+j\choose i}u_{i+j}$, is a divided power algebra. Thus $H^*(\Omega\Sigma S^2)\cong H_*(\C P^{\infty})$.
\end{enumerate}

\textbf{V.}

$H_*(\Omega\Sigma(\Omega\Sigma S^2))\cong\mathbb Z\langle w_1,w_2,\dots\rangle$. It is a free associative Hopf algebra with the comultiplication
\begin{equation}\label{F-2}
\Delta w_n=\sum\limits_{k=0}^n{n\choose k}w_k\otimes w_{n-k}.
\end{equation}

\textbf{VI.}

$H_*\left(\Omega\Sigma\left(\bigvee_{i=1}^\infty S^{2i}\right)\right)\cong\mathbb{Z}\langle \xi_1,\xi_2,\ldots\rangle$. It is a free associative algebra and has the structure of a graded Hopf algebra with the comultiplication
\begin{equation}\label{F-3}
\Delta\xi_n = 1\otimes \xi_n + \xi_n\otimes 1.
\end{equation}

\begin{lem}
\begin{itemize}
\item[(i)] Over the rationals, $\NSymm\otimes\Q\cong H_*( \Omega\Sigma\C P^\infty; \Q)$ is isomorphic to a Lie-Hopf algebra. 
\item[(ii)] Over the integers, $\NSymm\cong H_*(\Omega\Sigma\C P^\infty)$ is not isomorphic to a Lie-Hopf algebra.
\end{itemize}
\end{lem}
\begin{proof}
\begin{itemize}
\item[(i)] Let $f\colon \Omega\Sigma S^2\lra K(\coefZ, 2)\simeq\C P^\infty$ be the map which realises a generator of $H^2(\Omega\Sigma S^2)\cong\coefZ$. This map is a rational homotopy equivalence. As $H_*(\Omega\Sigma S^2;\Q)\cong\Q[w]$ with $\deg w=2$ is a Lie-Hopf algebra, the map $f_\Q$ induces a Hopf isomorphism between $\NSymm\otimes\Q\cong H_*( \Omega\Sigma\C P^\infty; \Q)$ and the Lie-Hopf algebra  $H_*(\Omega\Sigma S^2;\Q)$. 
\item[(ii)] Following the first example in Section 3.1, we conclude that $\NSymm\cong H_*(\Omega\Sigma\C P^\infty)$ is not isomorphic to a Lie-Hopf algebra.
\end{itemize}
\end{proof}
We will use topological methods to find a maximal subalgebra of the algebra $H_*(\Omega\Sigma\C P^\infty)$ which over $\Q$ is isomorphic to $H_*(\Omega\Sigma\C P^\infty; \Q)$ but over $\coefZ$ is a Lie-Hopf algebra.
\begin{thm}
The Hopf algebra $H_*(\Omega\Sigma\Omega\Sigma S^2)$ is a maximal subHopf algebra of $H_*(\Omega\Sigma\C P^\infty)$ which is isomorphic to a Lie-Hopf algebra.
\end{thm}
\begin{proof}
Let us start with  the map $f\colon\Omega\Sigma S^2\lra \C P^\infty$ which realises a generator of $H^2(\Omega\Sigma S^2)\cong\coefZ$ and which induces an isomorphism in rational homology but which obviously does not induce an isomorphism over the integers (see examples I (1) and IV (1)). By taking the loop suspension of $f$, we get a loop map $\Omega\Sigma f\colon\Omega\Sigma\Omega\Sigma S^2\lra \Omega\Sigma\C P^\infty$ which in rational homology induces a Hopf algebra isomorphism. Using the James-Hopf invariants, we produce a homotopy equivalence between $\Sigma\Omega\Sigma S^2$ and $\Sigma\bigvee_{i=1}^\infty S^{2i}$ showing that $\Omega\Sigma S^2$ is an $S^1S$-space. Now by  Theorem~\ref{coH}, we have that $H_*(\Omega\Sigma\Omega\Sigma S^2)$ is isomorphic as a Hopf algebra to a Lie-Hopf algebra, which finishes the proof.
\end{proof}

\subsection{Obstructions to desuspending a map}
In this subsection we explicitly write obstructions to desuspending the homotopy equivalence
\begin{equation}
\label{equivalence}
\Sigma\left(\Omega\Sigma S^2\right)\lra \Sigma\left(\bigvee_{i=1}^\infty S^{2i}\right).
\end{equation}
The homotopy equivalence
\[
a \colon \Omega\Sigma\left(\bigvee_{i=1}^\infty S^{2n}\right)
\lra \Omega\Sigma(\Omega\Sigma S^2)
\]
induces an isomorphism of graded Hopf algebras
\[
a_* \colon \mathbb{Z}\langle \xi_1,\xi_2,\ldots\rangle \lra
 \mathbb{Z}\langle w_1,w_2,\ldots\rangle
 \]
and its algebraic form is determined by the conditions
\[
 \Delta a_*\xi_n = (a_* \otimes a_*)(\Delta\xi_n).
\]
For example, $a_*\xi_1 = w_1,\; a_*\xi_2 = w_2- w_1|w_1,\; a_*\xi_3
= w_3- 3w_2|w_1 + 2w_1|w_1|w_1$.

Thus using topological results we have obtained that two Hopf algebra structures on the free associative algebra with comultiplications (\ref{F-2}) and (\ref{F-3}) are isomorphic over
$\mathbb{Z}$.

This result is interesting from the topological point of view, since the elements $(w_n-a_*\xi_n)$ for $n\geq 2$ are obstructions to the desuspension of homotopy equivalence~\eqref{equivalence}.

\end{document}